\documentclass{amsart}

\usepackage{latexsym,amsfonts,amsmath,amssymb,mathrsfs}
\usepackage{amsmath}
\usepackage{amsthm}
\usepackage{amssymb}
\usepackage{amsfonts}
\usepackage{graphicx}
\usepackage{pstricks,pstricks-add}
\usepackage{pst-math,pst-xkey}
\usepackage{hyperref}
\usepackage{color}
\usepackage{setspace}

\reversemarginpar

\newtheorem{theo}{Theorem}
\newenvironment{theorem}{\vspace{4mm}\begin{theo}}{\end{theo}}
\newtheorem{lem}[theo]{Lemma}
\newenvironment{lemma}{\vspace{4mm}\begin{lem}}{\end{lem}}
\newtheorem{coro}[theo]{Corollary}

\newtheorem{rem}[theo]{Remark}

\newtheorem{prop}[theo]{Proposition}
\newenvironment{proposition}{\vspace{4mm}\begin{prop}}{\end{prop}}
\newtheorem{defi}{Definition}

\newtheorem{assump}{Assumption}

\newtheoremstyle{citing}{}{}{\itshape}{}{\bfseries}{.}%
 { }{\thmnote{#3}}
\theoremstyle{citing}

\newcommand{\gap}{{\rm gap}}
 \newcommand{\Borel}{\mathcal{B}}

\newcommand{\N}{\mathbb{N}}
\newcommand{\R}{\mathbb{R}}
\newcommand{\abs}[1]{\left\vert #1 \right\vert}	
\newcommand{\norm}[1]{\left\Vert #1 \right\Vert}	

\newcommand{\E}{\mathbb{E}}

\renewcommand{\a}{\alpha}

\newcommand{\eps}{\varepsilon}

\newcommand{\dint}{\text{\rm d}}

\allowdisplaybreaks

\newcommand{\mask}[1]{{}}

\newcommand{\dr}[1]{{\black #1}}

\begin{document}

\doublespacing

\title[Error bounds of MCMC]{Error bounds of MCMC 
for functions with unbounded stationary variance}

\author{Daniel Rudolf}
\address{Inst. f\"ur Math.\\ Universit\"at Jena\\
Ernst-Abbe-Platz 2\\ 07743 Jena\\ Germany}
\email{daniel-rudolf@uni-jena.de}

\author{Nikolaus Schweizer}
\address{Mercator School of Management\\ Universit\"at Duisburg-Essen\\
Lotharstr. 65\\ 47057 Duisburg\\ Germany}
\email{nikolaus.schweizer@uni-due.de }

\begin{abstract}
 We prove explicit error bounds for Markov chain Monte Carlo (MCMC)
 methods to compute expectations of functions with unbounded
 stationary variance. We assume that there is a $p\in(1,2)$ so that 
 the functions have finite 
 $L_p$-norm.
 For uniformly ergodic Markov chains we obtain error bounds with 
 the optimal order of convergence
 $n^{1/p-1}$ and if there exists a spectral gap we almost get the optimal order.
 Further, a burn-in period is taken into account and a recipe for choosing the
 burn-in is provided.
\end{abstract}


\maketitle

%
\section{Introduction}
Let $G$ be a metric space and let 
$\Borel(G)$ 
be the corresponding Borel $\sigma$-algebra.
We study the problem of computing an expectation of a measurable function, 
say $f\colon G \to \R$, with respect
to a probability measure $\pi$. 
Thus we want to know
\[
 \E_{\pi} (f) = \int_G f(x)\, {\rm d}\pi(x).
\]
We assume that the variance 
$\E_{\pi} (f^2) - \E_{\pi} (f)^2$
is not finite, but that there
is a $p\in (1,2)$ such that
\[
 \norm{f}_{p} = \left( \int_G \abs{f(x)}^p \,\dint\pi(x) \right)^{1/p} < \infty.
\] 
This is the case if $f$ has a singularity, 
e.g. $f(x)=|x|^{-d/2}$ and $\pi$ 
has a bounded strictly positive density over a compact convex set $G\subset \R^d$ with  $0\in G$. 
Here, $|\cdot|$ denotes the Euclidean norm in $\mathbb{R}^d$.
Another application with $G\subseteq \R^d$ is the computation of a 
$2$nd
moment of $\pi$,
say $f(x)=x_d^2$, when the 
$4$th
moment of $\pi$ is infinite due to heavy tails.

Our focus is on situations where $\pi$ and $f$ are complicated and thus Monte Carlo algorithms are applied.
Often one does not have an i.i.d. sequence of random variables with distribution $\pi$.
For instance, this is the case if $\pi$ is only known up to a normalizing constant. 
Such situations naturally arise in Bayesian Statistics and Statistical Physics.
Markov Chain Monte Carlo (MCMC) methods are a popular approach for overcoming this problem.

The basic idea of MCMC is to approximate $\pi$ using a Markov chain $(X_n)_{n\in \mathbb{N}}$ 
with transition kernel $K$ and initial distribution $\nu$. Here $\pi$ 
is the stationary and limit distribution.
Then, one approximates $\E_\pi(f)$ by
\begin{equation}
  S_{n,n_0}(f) = \frac{1}{n} \sum_{j=1}^n f(X_{j+n_0}),
 \end{equation}
 where $n$ denotes the number of function evaluations and 
$n_0$ the burn-in. The burn-in is the number of 
steps needed to get sufficiently close to $\pi$.

By an ergodic theorem, 
see \cite{AsGl11} or \cite[Theorem~17.1.7, p.~427]{MeTw09},
the MCMC method is well defined, i.e.
for a $\varphi$-irreducible and Harris recurrent Markov chain 
$
 \lim_{n\to \infty} S_{n,n_0}(f) = \mathbb{E}_\pi(f)
$
holds almost surely for any $f$, with finite $\norm{f}_{1}$.
For $p\in[2,\infty]$, the ergodic theorem is usually augmented 
by a central limit theorem, for a survey see \cite{Jo04} and  
for estimating with confidence see \cite{FlJo11}.
The corresponding limit theorems for 
$p\in[0,2)$ have been less studied, see 
\cite{CaMa13} for recent results and an overview.

Under different convergence assumptions 
on the Markov chain, various non-asymptotic bounds on 
the mean square error of $S_{n,n_0}$ are known, see for example
\cite{BeCh09,
JoOl10,LaMiNi13,LaNi11,Ru09,Ru12}.
In particular, for $p\in[2,\infty]$ with finite $\norm{f}_{p}$ 
in \cite[Theorem~3.34 and Theorem~3.41]{Ru12} explicit error
bounds of $S_{n,n_0}$ are provided.
However, no explicit error bounds are known for $p\in(1,2)$. 
The present paper closes this gap.

For $p\in(1,2]$ we prove bounds on the absolute mean error 
\[
e_1(S_{n,n_0},f) = \E \abs{S_{n,n_0}(f)-\E_\pi(f)},
\]
for functions $f$, with $\norm{f}_{p}<\infty$.
We consider the absolute error
since the root mean square error of $S_{n,n_0}$ is not
necessarily finite for $p\in(1,2)$.

It is known that any algorithm that uses only $n$ function values of $f$, i.e.
$
 A_n(f) = \phi(f(X_1),\dots,f(X_n))
$
for some $\phi\colon \R^n \to \R$ 
with an arbitrary, possibly random sample $X_1,\dots,X_n \in G$,
satisfies
\begin{equation}  \label{eq: opt_order}
\sup_{\norm{f}_{p} \leq 1} e_1(A_{n},f) \geq c\cdot n^{1/p-1},
\end{equation}
for $p\in(1,2]$ and some number $c>0$.
For a proof of this fact follow the arguments of
\dr{\cite[Theorem~5.3]{He94} or \cite[Section~2.2.9, Proposition~1 with $k=0$]{No88}}.
Therefore, our bounds cannot decay faster than $n^{1/p-1}$, 
the optimal order of convergence.

Our absolute mean error bounds satisfy the following
properties:
\begin{itemize}
 \item For reversible uniformly ergodic 
 Markov chains we obtain the optimal order
 of convergence $n^{1/p-1}$. For Markov chains with a spectral gap we come arbitrarily close to the optimal order.
 \item We quantify the penalty 
 that arises since the initial distribution 
 $\nu$ is not the stationary distribution. 
 This penalty appears in our bounds through $\log\norm{\frac{d \nu}{d \pi}-1}_{\infty}$, where $\frac{d \nu}{d \pi}$
 denotes the density of $\nu$ with respect to $\pi$. 
 This effect is controlled 
 by the choice of the burn-in $n_0$.
We provide a recipe for choosing $n_0$.
\end{itemize}
The main idea of proof is simple and adapted from \cite[Proposition~5.4]{He94}.
The key technique 
is to apply the interpolation theorem 
of Riesz-Thorin to the absolute mean error and root mean square error, 
viewed as operators. 

\section{Main results}
This section summarizes our main results which are proved in Section~\ref{sec: prel}.
Let $(X_n)_{n\in\mathbb{N}}$ be a 
Markov chain with transition kernel $K$ and initial distribution $\nu$.
We always assume that $\pi$ is a stationary distribution.

By $L_2=L_2(\pi)$ we denote the set of all square integrable functions
with respect to $\pi$. For $f\in L_2$ note that the transition kernel
$K$ induces the Markov operator
\begin{equation}  \label{eq: MO}
  Pf (x) = \int_G f(y)\, K(x,\dint y). 
\end{equation}
We denote the
spectral gap of the Markov operator by 
\[
  \gap(P)=1-\norm{P-\E_{\pi}}_{L_2 \to L_2} 
\]
Now we state the error bound for reversible uniformly ergodic Markov chains.

\begin{theorem}  \label{thm: error_bound_uniform}
 Let us assume that we have a reversible Markov chain with transition kernel $K$
 and initial distribution $\nu$. Let $K$ be uniformly ergodic, 
 i.e. for an $\alpha \in [0,1)$ and an $M\in (0,\infty)$ it
 holds for $\pi$-almost all $x\in G$ that
 \begin{equation}  \label{eq: unif_erg}
    \norm{K^n(x,\cdot)-\pi}_{\mbox{\rm tv}} 
  \leq \a^n M,
 \end{equation}
 where $\norm{\cdot}_{\mbox{\rm tv}}$ denotes the total variation distance.	
 Further, assume that there exists $\frac{d\nu}{d\pi}$, 
 with finite $\norm{\frac{d\nu}{d\pi}}_{\infty}$, where $\frac{d\nu}{d\pi}$ 
 denotes the density of $\nu$ with respect to $\pi$.
 Let $p\in (1,2]$ and assume that $n_0 \in \N_0$ satisfies 
  \[
   n_0 \geq \frac{\log(2\,M \norm{\frac{d\nu}{d\pi}-1}_{\infty})}{ 1-\alpha}.
  \]
 Then
  \[
   e_1(S_{n,n_0},f) \leq 
   \frac{4\norm{f}_{p}}{ (n\cdot \gap(P))^{1-1/p}} +
    \frac{4 \norm{f}_{p}}{ (n\cdot(1-\alpha) )^{2-2/p}}.
    \]
\end{theorem}

First, note that uniform ergodicity 
implies $\gap(P)\geq 1-\alpha$.
The upper bound might be interpreted as follows: The burn-in 
$n_0$ is used to decrease the influence of the initial distribution.
The number $n$ decreases the error of the averaging procedure. The leading
term has the optimal order of convergence $n^{1/p-1}$, see \eqref{eq: opt_order}. 
The spectral gap appears in the leading term and $1-\a$ appears in the
higher order term. Both quantities describe the price we pay for 
using Markov
chains for approximate sampling.
If one can sample with respect to $\pi$, then $\a=0$, $\gap(P)=1$ and $\nu=\pi$.
Thus $n_0=0$ and the error bound is, up to a constant factor, 
the same as in \cite[Proposition~5.4]{He94}.

Now we state the result for Markov chains with a spectral gap.
Note that here we do not assume that the Markov chain is reversible.
 \begin{theorem}  \label{thm: error_bound_spectral_gap}
 Let us assume that we have a Markov chain with transition kernel $K$ 
 and initial distribution $\nu$. Let $\gap(P)>0$ 
 and further assume that there exists $\frac{d\nu}{d\pi}$, 
 with finite $\norm{\frac{d\nu}{d\pi}}_{\infty}$, where $\frac{d\nu}{d\pi}$ 
 denotes the density of $\nu$ with respect to $\pi$.
 Let $\delta\in (0,1]$, $p\in (1+\delta,2]$ 
 and assume that $n_0 \in \N_0$ satisfies 
  \[
   n_0 \geq 
 \frac{\log(64\, \delta^{-1} \norm{\frac{d\nu}{d\pi}-1}_{\infty}) 
  }{\delta\cdot \gap(P)}.
  \]
 Then
  \[
   e_1(S_{n,n_0},f) \leq 
    \frac{8\norm{f}_{p}}{ (n\cdot \gap(P))^{1-\frac{1+\delta}{p}}} +
    \frac{8 \norm{f}_{p}}{ (n\cdot \gap(P))^{2-\frac{2(1+\delta)}{p}}}.
    \]
\end{theorem}

 Let us interpret the result. 
 The burn-in $n_0$ is used to
 decrease the dependence on the initial distribution and $n$ denotes
 the sample size of the average procedure.
 The convergence of the Markov chain is captured by the spectral gap.
 However, an additional parameter $\delta \in (0,1]$ appears. This parameter
 measures a minimal integrability and provides a relation between integrability
 and convergence of the Markov chain. 
 It is fair to ask whether one can remove the additional parameter $\delta$. 
 The reason for the $\delta$
 lies in the mean square error bounds 
 for $L_p$-functions that enter our proofs via the Riesz-Thorin theorem. 
 These bounds deteriorate as $p$ approaches 2, 
 see Proposition~\ref{prop: mse_geometric_erg} below.
 With a mean square error bound for $L_2$-functions
 one could achieve $\delta=0$. 
 To our knowledge, such bounds are not known,
 even under the additional assumption of reversibility.

 For $\delta$ close to zero, the rate of convergence  in the error bound is arbitrarily close to
 optimal. 
 But we pay a price. Namely, the 
 burn-in $n_0$ increases for decreasing $\delta$. 
 There is thus
 a trade-off in determining $\delta$ and one might ask for an optimal $\delta$.
 After some computations by hand one can guess that
 \begin{equation} \label{eq: heuristic_delta}
     \hat \delta = \frac{\sqrt{p-1}}{\sqrt{p}} 
   \left( \frac{ \log(64\norm{\frac{d\nu}{d\pi}-1}_{\infty})}
   {(16\,\eps^{-1})^{p/(p-1)} \log(16\,\eps^{-1})} \right)^{1/2},
 \end{equation}
 is a good choice for $\delta$ to achieve an error estimate smaller 
 than $\varepsilon$. We justify this heuristic $\hat \delta$ as follows:
 For different values of $p$ and $\eps$ we 
 numerically compute $\delta^*$ which minimizes
 the total size of the Markov chain sample $N(\delta)=n(\delta)+n_0(\delta)$ which is needed to obtain
 an estimate with error $\eps$ from Theorem~\ref{thm: error_bound_spectral_gap}. 
 In Table~\ref{tab: eps_fest} and Table~\ref{tab: eps_variabel} one can see
 that $\hat \delta$ and $\delta^*$ have the same behavior for decreasing $\eps\in(0,1]$ 
 and decreasing $p\in(1,2)$. Furthermore, the numbers $N(\hat \delta)$ and $N(\delta^*)$
 are quite close.
 For $p$ close to $1$ (Table~\ref{tab: eps_fest}) we see that
the penalty for the lack of integrability 
leads to a drastic increase in $N(\delta^*)$. 
This is not surprising, since for $p$ close to $1$ 
Theorem~\ref{thm: error_bound_uniform}
exhibits a similar behavior.
However, for $p$ not too far away from $2$, the total size
of the Markov chain sample $N(\delta^*)$ and $N(\hat \delta)$ is reasonable.

 \begin{table} 
 \begin{center}
 \begin{tabular}{|c|c|c| c| c|}
 \hline \vspace{-9pt} & & & & \\
 $p$ & $\delta^*$& $N(\delta^*)$
 & $\hat \delta$& $N(\hat \delta)$ \\
\hline \vspace{-11pt} & & & & \\
1.1 & $5.01 \cdot 10^{-11}$ &$9.83 \cdot 10^{22}$  
&$8.64 \cdot 10^{-13} $&$9.83 \cdot 10^{22}$ \\
\hline \vspace{-11pt} & & & & \\
1.3 & $8.39 \cdot 10^{-5}$  & $1.88 \cdot 10^{10} $
& $3.06 \cdot 10^{-5}$ & $1.89 \cdot 10^{10}$\\
\hline \vspace{-11pt} & & & & \\
1.5 & $2.31  \cdot 10^{-3}$  &$ 5.99 \cdot 10^7$  
& $1.08  \cdot 10^{-3}$ & $6.21 \cdot 10^7$\\
  \hline 
 \end{tabular}
 \end{center}
 \caption{
 Size of the Markov chain sample $N(\delta)$ 
 required by Theorem 2 for different values of $p$ and $\delta \in \{\delta^*,\hat{\delta}\}$. 
 The parameters are $\norm{f}_{p} = 1$, $\gap(P)=0.01$,
  $\norm{\frac{d\nu}{d\pi}-1}_{\infty}=10^{30}$ and $\eps=0.1$.}
\label{tab: eps_fest}
  \end{table}

  \begin{table}  
 \begin{center}
 \begin{tabular}{|c|c|c| c| c|c|c|c|c|c|c|}
 \hline \vspace{-9pt} & & & & \\
 $\varepsilon$ & $\delta^*$& $N(\delta^*)$
 & $\hat \delta$& $N(\hat \delta)$ \\
\hline \vspace{-11pt} & & & &\\
0.01 & $4.90  \cdot 10^{-7}$  &$ 3.82 \cdot 10^{14}$  
& $1.73  \cdot 10^{-7}$ & $3.82 \cdot 10^{14}$\\
  \hline \vspace{-11pt} & & & &\\
0.2 & $3.92  \cdot 10^{-4}$  &$ 1.01 \cdot 10^9$  
& $1.48  \cdot 10^{-4}$ & $1.04 \cdot 10^9$\\
  \hline \vspace{-11pt} & & & &\\
0.5 & $2.85  \cdot 10^{-3}$  &$ 2.66 \cdot 10^7$  
& $1.21  \cdot 10^{-3}$ & $2.89 \cdot 10^6$\\
\hline
   \end{tabular}
 \end{center}
\caption{
 Size of the Markov chain sample $N(\delta)$ 
 required by Theorem 2 for different values of $\eps$ and $\delta \in \{\delta^*,\hat{\delta}\}$. 
 The parameters are $\norm{f}_{p} = 1$, $\gap(P)=0.01$,
  $\norm{\frac{d\nu}{d\pi}-1}_{\infty}=10^{30}$ and $p=1.3$.}
\label{tab: eps_variabel}
	   \end{table}

\section{Auxiliary results and proofs} \label{sec: prel}

By \eqref{eq: MO} the transition kernel $K$ induces the Markov operator $P$ acting
on functions and by
\[
 \nu P(A) = \int_G K(x,A)\, {\rm d}\nu(x), \quad A \in \Borel(G),
\]
it induces the Markov operator acting on signed
measures $\nu$ on $(G,\Borel(G))$. If $\nu$ is absolutely continuous 
with respect to $\pi$, 
then also $\nu P$ is absolutely continuous with respect to $\pi$.
In particular, 
$\frac{d (\nu P)}{d \pi} = P^* (\frac{d\nu}{d\pi})$ 
with the adjoint operator $P^*$,
for details see \cite[Lemma~3.9]{Ru12}.

For $p\in[1,\infty]$ we denote by $L_p=L_p(\pi)$ the space of all functions 
$f\colon G \to \R$ with $\norm{f}_{p} < \infty$. 
Note that
$\norm{P}_{L_p\to L_p}=1$
and that $P\colon L_2 \to L_2$ 
is self-adjoint whenever the Markov chain is reversible.

Now we define a generalized error term of $S_{n,n_0}$ with parameter $p\in [1,2]$ 
for the computation of $\E_{\pi}(f)$. Let
\[
e_p(S_{n,n_0},f) := \left(\E \abs{S_{n,n_0}(f)-\E_{\pi}(f)}^p\right)^{1/p}.
\]
Note that for $p=1$ this is the absolute mean error and for $p=2$ we have
the root mean square error. 
The expectation in the definition of the error is taken 
with respect to the 
distribution, say $\mu_{\nu,K}$, 
of the trajectory $X_1,\dots,X_{n+n_0}$.


\subsection{Proof of Theorem~\ref{thm: error_bound_uniform}}
We prove that under the assumptions of Theorem~\ref{thm: error_bound_uniform} 
the inequality
\begin{equation}  \label{eq: to_prove_unif_erg}
 \begin{split}
   \sup_{\norm{f}_{p}\leq 1} e_p(S_{n,n_0}, f)
& \leq 2 ^{2/p-1} \left( 1+ 2 \alpha ^{n_0}  M \norm{\frac{d\nu}{d\pi}-1}_{\infty}\right)^{2/p-1} \\
& \; \times \left( \frac{2^{1-1/p}}{(\,n\cdot \gap(P)\, )^{^{1-1/p}}} 
+ \left( \frac{4\, M \norm{\frac{d\nu}{d\pi}-1}_{\infty} \alpha^{n_0}}{n^2(1-\alpha)^2} \right)^{1-1/p} \right)
 \end{split}
\end{equation}
holds for $p\in (1,2]$. From this upper bound the assertion of the theorem
follows immediately by $\log \a^{-1} \geq 1-\a$ for $\a \in [0,1]$, by the choice of the burn-in $n_0$ and since $e_1(S_{n,n_0},f) \leq e_{p}(S_{n,n_0},f)$.
We begin with two auxiliary inequalities. 
Proposition \ref{prop: mse_uniform_erg} provides an upper bound on 
the root mean square error for $f\in L_2$, 
see \cite[Theorem~3.34]{Ru12}. 
Lemma \ref{lem: abs_err} states that the absolute 
mean error is bounded for $f\in L_1$.

\begin{proposition} \label{prop: mse_uniform_erg}
Under the assumptions of Theorem~\ref{thm: error_bound_uniform} we have
\[
 \sup_{\norm{f}_2 \leq 1} e_2 (S_{n,n_0},f) 
 \leq \frac{\sqrt{2}}{(\,n\cdot \gap(P)\, )^{1/2}} 
  + \frac{2\, M^{1/2} \norm{\frac{d\nu}{d\pi}-1}_{\infty}^{1/2} \alpha^{n_0/2} }{n(1-\alpha)}. 
\]
\end{proposition}
\begin{lemma} \label{lem: abs_err}
Under the assumptions of Theorem~\ref{thm: error_bound_uniform} we have
\[
 \sup_{\norm{f}_1\leq 1 } e_1(S_{n,n_0}, f) 
 \leq 2  + 4 \alpha^{n_0} M \norm{\frac{d\nu}{d\pi}-1}_{\infty}.
\]
 \end{lemma}
\begin{proof}
By the uniform ergodicity we have with \cite[Proposition~3.24]{Ru12}
\begin{equation} \label{eq: L_infty_erg}
  \norm{P^n-\E_{\pi}}_{L_\infty \to L_\infty} \leq \a^n 2M, \quad n\in \mathbb{N}.
\end{equation}
With the adjoint operator $P^*$ of $P$ and by
$\frac{d (\nu P)}{d \pi} = P^* (\frac{d\nu}{d\pi})$, see \cite[Lemma~3.9]{Ru12},
we obtain
\begin{equation}  \label{eq: later}
 \begin{split}
   & \quad\; e_1(S_{n,n_0}, f) 
 \leq \frac{1}{n} \sum_{j=1}^n \mathbb{E} \abs{f(X_{j+n_0})-\mathbb{E}_{\pi}(f)} \\
& = \frac{1}{n} \sum_{j=1}^n \int_G \abs{f(x)-\mathbb{E}_{\pi}(f)} (P^*)^{j+n_0}\left(\frac{d \nu}{d\pi}\right)(x)\, {\rm d}\pi(x)\\
& = \frac{1}{n} \sum_{j=1}^n 
\left(\norm{f-\mathbb{E}_{\pi}(f)}_{1} 
+ \int_G \abs{f(x)-\mathbb{E}_{\pi}(f)} ((P^*)^{j+n_0}-\E_{\pi})\left(\frac{d \nu}{d\pi}-1\right)(x)\, {\rm d}\pi(x)\right).
 \end{split}
\end{equation}
\dr{Further, $\norm{f-\E_\pi(f)}_1 \leq 2 \norm{f}_1$ and}
by the assumed reversibility we have $P=P^*$ which leads to
\begin{align*}
e_1(S_{n,n_0}, f) 
\underset{\eqref{eq: L_infty_erg}}{\leq}  2 \norm{f}_{1}  \left( 1+ 2 \alpha^{n_0} M \norm{\frac{d \nu}{d\pi}-1}_{\infty}  \right).
 \end{align*}

 \end{proof}
Now we prove \eqref{eq: to_prove_unif_erg}. 
Recall that $\mu_{\nu,K}$ denotes the distribution of the sample trajectory and
consider the linear operator $T\colon L_{p}(\pi) \to L_p(\mu_{\nu,K})$
given by
\begin{equation}  \label{eq: op_interpol}
 T(f) = S_{n,n_0}(f)-\E_{\pi}(f).
\end{equation}
Further, note that
$
 \norm{T}_{L_p(\pi) \to L_p(\mu_{\nu,K})} = \sup_{\norm{f}_{p}\leq 1} e_p(S_{n,n_0},f).
$
By Proposition~\ref{prop: mse_uniform_erg} and Lemma~\ref{lem: abs_err} we obtain
 $$\norm{T}_{L_1(\pi)\to L_1(\mu_{\nu,K})} \leq M_1 
\;\;\text{ and}\;\; 
  \norm{T}_{L_2(\pi)\to L_2(\mu_{\nu,K})} \leq M_2, $$
with 
\begin{align*}
  M_1 & = 2+ 4 \alpha^{n_0} M \norm{\frac{d\nu}{d\pi}-1}_{\infty},\\
  M_2 & =  \frac{\sqrt{2}}{(\,n\cdot \gap(P)\, )^{1/2}} + \frac{2\, M^{1/2} \norm{\frac{d\nu}{d\pi}-1}_{\infty}^{1/2} \alpha^{n_0/2}}{n(1-\alpha)}.
\end{align*}
The application of Proposition~\ref{prop: riesz_thorin} (Riesz-Thorin theorem)
leads to
$\norm{T}_{L_p(\pi) \to L_p(\mu_{\nu,K})} \leq M_1 ^{1-\theta} M_2^\theta$ 
with $\theta= 2-2/p$ and
\eqref{eq: to_prove_unif_erg} now follows
by 
$(x+y)^{r} \leq x^r + y^r$ for $x,y\geq 0$ and $r\in [0,1]$.

\subsection{Proof of Theorem~\ref{thm: error_bound_spectral_gap}}
We prove that for any $\delta\in(0,1]$ and $p\in[1+\delta,2]$
under the assumptions of Theorem~\ref{thm: error_bound_spectral_gap} 
\begin{equation}  \label{eq: to_prove_gap}
 \begin{split}
&   \sup_{\norm{f}_{p}\leq 1} e_1(S_{n,n_0}, f)
\leq 2\left( 2+ 4 (1-\gap(P))^{2 \frac{n_0\delta}{1+\delta}} \norm{\frac{d\nu}{d\pi}-1}_{\infty}\right)^{2\frac{1+\delta}{p}-1} \\
& \quad\; \times \left( \frac{2^{1-\frac{1+\delta}{p}}}{(\,n\cdot \gap(P)\, )
  ^{1-\frac{1+\delta}{p}}} 
  +  \frac{\left(64\,
   \frac{1+\delta}{\delta} 
  \norm{\frac{d\nu}{d\pi}-1}_{\infty} 
 (1-\gap(P))^{2 \frac{n_0 \delta}{1+\delta}}\right)^{1-\frac{1+\delta}{p}}}{\left(n^2\cdot\gap(P)^2\right)^{1-\frac{1+\delta}{p}} }
\right).
 \end{split}
\end{equation}
From this upper bound and the choice
\[
 n_0 \geq \frac{1+\delta}{2\delta}\cdot \frac{\log(\frac{32(1+\delta)}{\delta} \norm{\frac{d\nu}{d\pi}-1}_{\infty})}{\log(1-\gap(P))^{-1}}
\]
the assertion of the theorem follows 
by taking $\log(1-\gap(P))^{-1} \geq \gap(P)$ and $\delta \in (0,1]$ into 
account.
We next state two auxiliary inequalities with parameters $p_1\in [1,2]$ and $p_2 \in (2,4]$: 
By Proposition \ref{prop: mse_geometric_erg} we have an upper bound on the root mean square error 
for $f\in L_{p_2}$, 
see \cite[Theorem~3.41]{Ru12}. Lemma \ref{lem: abs_err_geometric_erg}
states that the absolute mean error is bounded for $f\in L_{p_1}$.
\begin{proposition} \label{prop: mse_geometric_erg}
Under the assumptions of Theorem~\ref{thm: error_bound_spectral_gap} we have 
\begin{align*}
 \sup_{\norm{f}_{p_2}\leq 1 } e_2(S_{n,n_0}, f) &\leq 
\frac{\sqrt{2}}{(n\cdot \gap(P))^{1/2}}
+ \frac{8\sqrt{p_2}}{\sqrt{p_2-2}}
\cdot \frac{ \norm{\frac{d\nu}{d\pi}-1}_{\infty}^{1/2}  
(1-\gap(P))^{n_0 \frac{p_2-2}{p_2}} }
{n\cdot \gap(P) }. 
\end{align*}
\end{proposition}
\begin{lemma} \label{lem: abs_err_geometric_erg}
Under the assumptions of Theorem~\ref{thm: error_bound_spectral_gap} we have 
\begin{align*}
 \sup_{\norm{f}_{p_1}\leq 1 } e_1(S_{n,n_0}, f) 
  &\leq 2 + 4 \, \norm{\frac{d\nu}{d\pi}-1}_{\infty} 
 (1-\gap(P))^{2\frac{p_1-1}{p_1}\, n_0}.
\end{align*}
 \end{lemma}
\begin{proof}
For $p_1\in[1,2]$ and $n\in\N$ we obtain by \dr{$P^n-\E_\pi = (P-\E_\pi)^n$} and by
Proposition~\ref{prop: riesz_thorin} 
with $\norm{P^n-\mathbb{E}_\pi}_{L_1 \to L_1} \leq 2$ 
that 
\begin{equation}  \label{eq: 11}
 \norm{P^{n}-\mathbb{E}_\pi}_{L_{p_1}\to L_{p_1}} 
 \leq 2^{2/p_1-1} \norm{P-\mathbb{E}_\pi}_{L_2\to L_2}^{2 \frac{p_1-1}{p_1}\, n}
 \leq  2 (1-\gap(P))^{2\frac{p_1-1}{p_1}\, n}.
\end{equation}
Furthermore, we have
$
\norm{(P^*)^n-\E_\pi}_{L_{\widetilde{p}_1} \to L_{\widetilde{p}_1}}
=\norm{P^n-\E_\pi}_{L_{p_1}\to L_{p_1}},
$
with  $\widetilde{p}_1=\frac{p_1}{p_1-1}$ such that $p_1^{-1}+\widetilde{p_1}^{-1}=1$.
For details we refer to \cite[p.~42]{Ru12}.

We follow the proof of Lemma~\ref{lem: abs_err} until the end 
of \eqref{eq: later}.
Then, by $\norm{f-\mathbb{E}_\pi(f)}_1 \leq 2 \norm{f}_1 \leq 2\norm{f}_{p_1}$
and
H\"{o}lder's inequality with parameters $p_1$ and $\widetilde{p_1}$ 
we obtain
\begin{align*}
e_1(S_{n,n_0}, f) 
& \leq  2 \norm{f}_{p_1}  
\left( 1+ \norm{(P^*)^{j+n_0}-\E_\pi}_{L_{\widetilde p_1} \to L_{\widetilde p_1}} 
\norm{\frac{d \nu}{d\pi}-1}_{{\widetilde p_1}}  \right) \\
& \underset{\eqref{eq: 11}}{\leq} 2 \norm{f}_{p_1}  
\left( 1+ 2 \norm{\frac{d \nu}{d\pi}-1}_{\infty} 
(1-\gap(P))^{2\frac{p_1-1}{p_1}\, n_0}\right).
 \end{align*}

\end{proof}
Relying on Proposition~\ref{prop: mse_geometric_erg} and Lemma~\ref{lem: abs_err_geometric_erg} we can apply similar
interpolation arguments as in the proof
of Theorem~\ref{thm: error_bound_uniform}. 
We obtain the following:
\begin{lemma} \label{prop: interpolation_geometric}
  Let $p_1 \in [1,2]$, $p_2\in (2,4]$ and $p \in [p_1,p_2]$.
  Then, under the assumptions of Theorem~\ref{thm: error_bound_spectral_gap} we have 
\begin{align*}
 \sup_{\norm{f}_p\leq 1 } e_q(S_{n,n_0}, f) &\leq 2\,
 M_1^{\frac{p_1}{p_2-p_1}(\frac{p_2}{p}-1)} \cdot  M_2^{\frac{p_2}{p_2-p_1}(1-\frac{p_1}{p})}
\end{align*}
with
$ q  =	1+\frac{p_2(p-p_1)}{p_2(p+p_1)-2p p_1}  \in [1,2]$ and
\begin{align*}
      M_1 & = 2  + 4 \, \norm{\frac{d\nu}{d\pi}-1}_{\infty} (1-\gap(P))^{2\frac{p_1-1}{p_1}\, n_0},\\   
      M_2 & = \frac{\sqrt{2}}{(n\cdot \gap(P))^{1/2}}
+ \frac{8\sqrt{p_2}}{\sqrt{p_2-2}}
\cdot \frac{ \norm{\frac{d\nu}{d\pi}-1}_{\infty}^{1/2}  
(1-\gap(P))^{n_0 \frac{p_2-2}{p_2}} }
{n\cdot \gap(P) }.
 \end{align*}
 \end{lemma}
 \begin{proof}
 Consider the linear operator $T\colon L_{p}(\pi) \to L_q(\mu_{\nu,K})$
 from \eqref{eq: op_interpol}.
Note that
$
 \norm{T}_{L_p(\pi) \to L_q(\mu_{\nu,K})} = \sup_{\norm{f}_{p}\leq 1} e_q(S_{n,n_0},f).
$
By Lemma~\ref{lem: abs_err_geometric_erg} and Proposition~\ref{prop: mse_geometric_erg}  
we have
\begin{align*}
 \norm{T}_{L_{p_1}(\pi)\to L_1(\mu_{\nu,K})} \leq M_1 
\quad \mbox{and} \quad
  \norm{T}_{L_{p_2}(\pi)\to L_2(\mu_{\nu,K})} \leq M_2. 
\end{align*}
By Proposition~\ref{prop: riesz_thorin} (Riesz-Thorin theorem)
$\norm{T}_{L_p(\pi) \to L_q(\mu_{\nu,K})} \leq 2 M_1 ^{1-\theta} M_2^\theta$ holds for $\theta\in [0,1]$ satisfying $q^{-1}=1-\frac{\theta}2$ and
$p^{-1}=(1-\theta){p_1}^{-1}+ \theta {p_2}^{-1}$, i.e.
$\theta={\frac{p_2}{p_2-p_1}(1-\frac{p_1}{p})}$.
 \end{proof}
Since $e_1(S_{n,n_0},f) \leq e_{q}(S_{n,n_0},f)$ for $q\geq 1$, 
\eqref{eq: to_prove_gap} follows by an application
of Lemma~\ref{prop: interpolation_geometric}
with $p_1=1+\delta$ and $p_2 = 2 (1+\delta)$. 
 
\begin{appendix}
 \section{Riesz-Thorin interpolation theorem}	
 Let $(G,\mathcal{G},\pi)$ and $(\Omega, \mathcal{F},\mu)$ be probability spaces. 
  Let $p\in[1,\infty]$ and let $L_p(\pi)$ be the space of $\mathcal{G}$-measurable functions
  $g\colon G \to \mathbb{R}$ with
  $
    \norm{g}_{p,\pi} = 
    \left(\int_{G} \abs{g(x)}^p {\rm d}\pi(x) \right)^{1/p} < \infty 
  $
  and let $L_p(\mu)$ be the space of $\mathcal{F}$-measurable functions 
  $f\colon \Omega \to \mathbb{R}$
  with  
  $
     \norm{f}_{p,\mu} = \left(\int_{\Omega} \abs{f(x)}^p {\rm d}\mu(x) \right)^{1/p} < \infty.
  $
In the following we formulate a version of the theorem of Riesz-Thorin. 
For details we refer to \cite[Chapter~4: Corollary~1.8, Excercise~5, Corollary~2.3]{BeSh88}.
\begin{prop}[Riesz-Thorin theorem] \label{prop: riesz_thorin}
Let $1\leq p_k \leq  q_k \leq \infty$ for $k=1,2$. 
We assume that $\theta \in [0,1]$ and 
\[
\frac{1}{p} = \frac{1-\theta}{p_1}+\frac{\theta}{p_2}, \qquad
\frac{1}{q} = \frac{1-\theta}{q_1}+\frac{\theta}{q_2}.
\]
Let $T$ be a linear operator from $L_{p_1}(\pi)$ to $L_{q_1}(\mu)$ 
and at the same time from $L_{p_2}(\pi)$ to $L_{q_2}(\mu)$ with
\begin{align*}
\norm{T}_{L_{p_1}(\pi) \to L_{q_1}(\mu)} \leq M_1, \quad\quad 
\norm{T}_{L_{p_2}(\pi) \to L_{q_2}(\mu)} \leq M_2.
\end{align*}
Then 
\begin{equation} \label{eq: oper_norm_interpol}
\norm{T}_{L_p(\pi)\to L_q(\mu)} \leq M_1^{1-\theta} M_2^\theta
\end{equation}
and if, $1\leq p_k, q_k \leq \infty$ with $k=1,2$, 
then \eqref{eq: oper_norm_interpol} is replaced by $\norm{T}_{L_p(\pi)\to L_q(\mu)} \leq 2 M_1^{1-\theta} M_2^\theta
$
.
\end{prop}  

\end{appendix}

\subsection*{Acknowledgement}
We thank Erich Novak and two anonymous referees for valuable comments. 
D.R. was supported by the DFG Priority Program 1324 and the
 DFG Research Training Group 1523. 
 N.S. was supported by the DFG Priority Program 1324.


%
%
%

\end{document}